\documentclass{article}
\usepackage[T2A]{fontenc}
\usepackage[cp1251]{inputenc} 
\usepackage[english]{babel}
\usepackage{amsfonts,amssymb,mathrsfs,amscd}
\usepackage{amsthm,graphicx}

\theoremstyle{definition}
\newtheorem{definition}{Definition}
\newtheorem{theorem}{Theorem}
\newtheorem{lemma}{Lemma}

\def\Z{{\mathbb Z}}
\def\R{{\mathbb R}}

\def\cC{{\mathbb C}}

\def\Q{{\mathbb Q}}

\title{Maps from knots to 2-component links,
chord diagrams, and a way to enhance
Vassiliev invariants.}
\author{Vassily Olegovich Manturov}

\begin{document}

\maketitle

\begin{abstract}
In the present paper, we discuss a way of generalising Vassiliev knot
invariants for knots
in $3$-manifolds and weight systems to
framed chord diagrams having framing
$0$ and $1$.
\end{abstract}

Keywords: knot, link, Vassiliev invariant,
chord diagram, weight system,
parity, Kontsevich integral,
satellite knot, 3-manifold, link homotopy

AMS MSC: 57K10,57K12,57K14,81T18


\section{Introduction}

It is well known that symbols of the
Vassiliev knot invariants give rise
to {\em weight systems}, i.e., 
linear functions on chord diagrams 
satisfying the 4-term and the 1-term
relation. There is a {\em framed}
counterpart of chord diagrams, 
i.e., diagrams having {\em odd}
and {\em even} chords, with 4T-relations
modified accordingly \cite{IM}.

%

Framed chord diagrams seem to be
more interesting and richer than
usual ones and it would be interesting
to apply them to some knot theories.

$\Z_{2}$-framings and parities seem
to be very important to construct 
a principally new sort of knot invariants,
{\em picture-valued invariants}, 
see \cite{Parity2010}.
In the case of virtual knots where
parity is highly non-trivial \cite{VKnots}, one
can construct lots of new invariants,
including Vassiliev invariants based
on parities of chords giving
yet another motivation of the present paper.

Also, framed chord diagrams naturally
appear in the study of Legendrian
knots, and study of plane curves \cite{Lando}.\footnote{Such curves may have
two types of tangencies which
can naturally correspond to 
0-framings and 1-framings of 
chord diagrams, see \cite{IM}.}

In the present paper we address the
question about the geometric meaning
of such relations and ways to
obtain parity for classical
objects.

We derive framed relations from the Vassiliev invariants
of knots in $3$-manifolds having
non-trivial first $\Z_{2}$-cohomology
and discuss how this techniques 
can be applied to classical knots
in $\R^{3}$.

Throughout the paper, all knots and links are assumed oriented, unless specified otherwise.

Note that present paper is expository
from the point of view of invariants.
Here we present a few theorems,
but the main goal is ``to pass from
the classical world (with usual
Vassiliev invariants, chord diagram algebra, weight systems), to the ``virtual world'' where all objects and invariants are more interesting and complicated.

Some possible further steps
are collected in the epilogue.

The paper is organised as follows.
In Section 2, we define various linear
spaces and algebras of chord diagrams\footnote{Actually, 
multiplication is well-defined
for usual (unframed) chord diagrams,
but not on framed ones; we shall discuss it later}.
In Section 3, we discuss Vassiliev's
invariants for knots in various
3-manifolds, in particular,
in complements to knots in $\R^{3}$.
Theorem \ref{mainthm} says
the symbol of any ``good'' Vassiliev invariant considered
as a function on framed chord diagrams
gives rise to a function satisfying the {\em framed $4T$-relation} (and $1T$-relation)
where framing comes from the $\Z_{2}$-homology
of the ambient manifold. It is
known that in $\R^{3}$ weight systems satisfying
the 1T-relation give rise to Vassiliev
invariants. This is realised
by the Kontsevich integral, which is then discussed in Section 4. 
We address the same question in the framed
case. In Sections 5 we discuss our further
strategy: we may take a knot $K$,
take its $2$-cable $K_{1}\sqcup K_{2}$
and then take the value of {\em some 
finite-type invariant $\phi$
of knots in $\R^{3}\backslash K_{2}$}.

The problem is that the for different
knots $K$ one should take {\em different
functions $\phi$ } (since the domain
of $\phi$ is $\R^{3}\backslash K_{2}$).

But we conjecture that from some weight systems
on framed chord diagrams one could be 
able to get invariants for all such
$\R^{3}\backslash K_{2}$ and consequently
get new invariants of ordinary knots
in $\R^{3}$ with some flavour of
framing and parity.

In Section 5
 we
formulate further research problems.

\subsection{Acknowledgements}

I am grateful to Seongjeong Kim, Igor Mikhailovich Nikonov, and Alina Vassilievna Pital 
for useful discussions.

I am extremely grateful to the referee
for lots of valuable remarks.

 This work  was funded by the RSF grant No. 25-21-00884.

\section{Chord diagrams and framed chord diagrams}

\begin{definition}
By a {\em chord diagram} on $n$ chords
we mean a graph on $2n$ vertices consisting
of one oriented Hamiltonian cycle (called {\em the core circle of the diagram})
and $n$ edges (called {\em chords}) so
that each vertex is incident to exactly one
chord.

\end{definition}

\begin{definition}
Let $S$ be a set.
We say that a chord diagram $C$ is {\em framed}
with framing $S$ if each chord of $C$ is associated
with an element from $S$.

Later on, we shall deal only with $\Z_{2}$-framings.
\end{definition}

By the {\em empty chord diagram} we mean one
circle with no chords. Chord diagrams are
graded by the number of chords.

Let $F$ be the ground field (we shall
mostly deal with $\R$, $\cC$ or $\Q$),
and let us consider the linear spaces
$C_{n}$ and $C^{f}_{n}$ generated by all chord diagrams and all
framed chord diagrams respectively.

Now, let us take the quotient
spaces of $C_{n}$ by the
$4T$-relation, see the top row
of Fig. \ref{Framed},
and the quotient space of
 (resp., $C^{f}_{n}$) by all 
4T-relations shown in Fig. \ref{Framed}).

\begin{figure}
\centering\includegraphics[width=200pt]{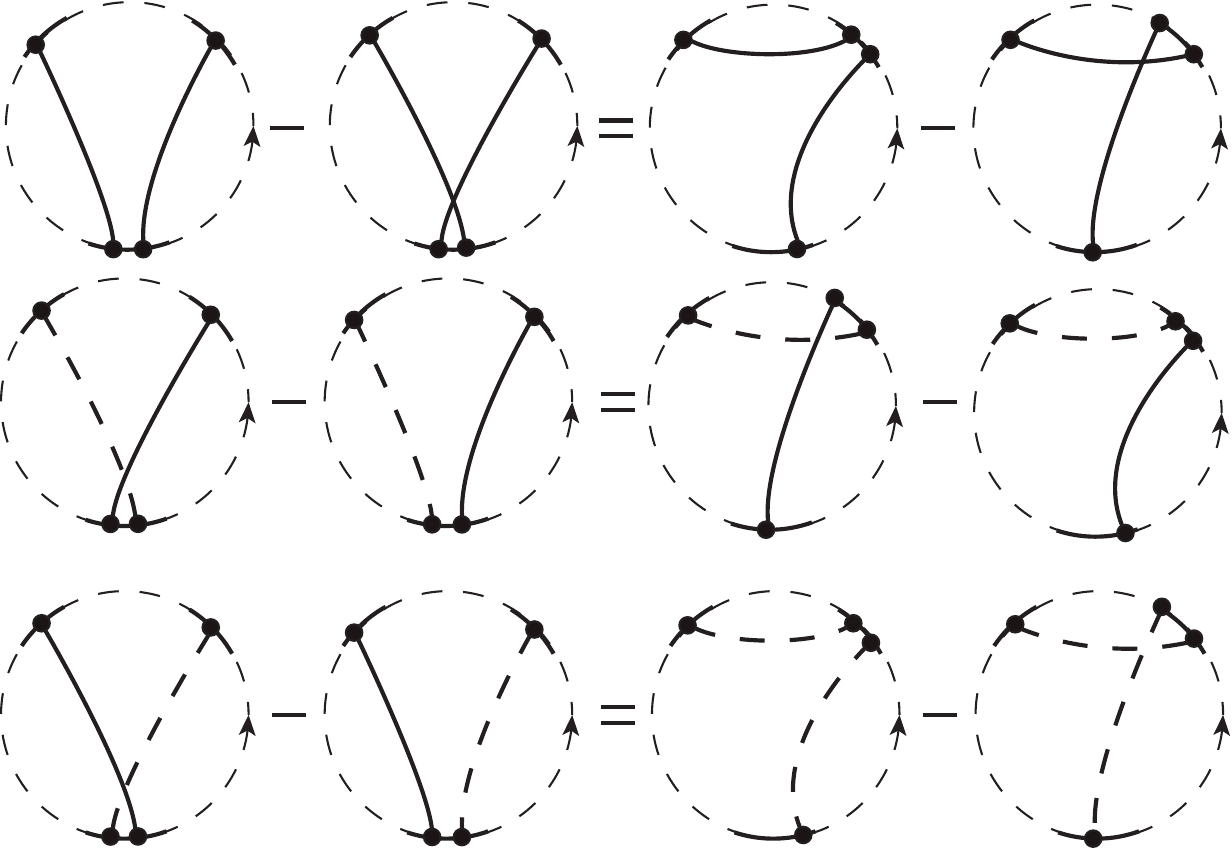}
\caption{The framed 4T-relations;
standard 4T-relation in the top row}
\label{Framed}
\end{figure}

We call these spaces ${\cal C}_{n}$
and ${\cal C}_{n}^{f}$, respectively.

Similarly, let us define the 
quotient space ${\cal C}^{f*}_{n}$
as the quotient space of $C_{n}^{f}$
by the relations where the
bottom row of Fig. \ref{Framed}
is replaced with Fig.\ref{framed4T}.
 Here solid chords 
correspond to framing 0; dashed chords
correspond to framing 1.

In total, we have $n$ chords where
only two of them are depicted in the
figure, and
$n-2$ of them are fixed and have
endpoints on dashed parts of the circle.

\begin{figure}
\centering\includegraphics[width=200pt]{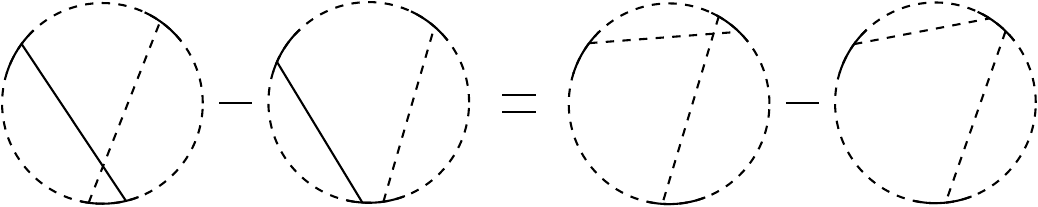}

\caption{The other version of
the last framed $4T$-relation}

\label{framed4T}
\end{figure}

\begin{lemma}
For each $n$, the spaces ${\cal C}_{n}^{f}$
and ${\cal C}_{n}^{f*}$ are naturally isomorphic. \label{lmm}
\end{lemma}

\begin{proof}
Indeed, the isomorphism takes
any chord diagram $c$ to $(-1)^{\downarrow}c$, where $\downarrow$ means the number
of odd chords of $c$.

One easily checks that  either RHS or LHS
of the framed $4T$-relation in Fig. \ref{framed4T} changes
its sign.

\end{proof}

\begin{definition}
Similarly to chord diagrams (framed or not) one can define {\em arc diagrams}
where instead of an oriented {\em circle} and {\em chords}
we deal with an oriented {\em line} and arcs
connecting points on it.
\end{definition}

For arc diagrams we similarly define
linear spaces $A,A^{f}$, and 
quotient spaces ${\cal A},{\cal A}^{f},
{\cal A}^{f*}$.

There is an obvious map, called 
{\em the closure} from the
linear space of arc diagrams
to the linear space of chord diagrams
which closes up the line to the circle.

The inverse map (breaking a circle
at a point which does not belong to
a chord) is not well defined: 
the result depends on the choice
of the breaking point (\cite{IM}).

The analogue of Lemma \ref{lmm}
works verbally for arc diagrams.

\begin{definition}
By the {\em one-term relation} on 
(framed) chord diagrams
we mean a relation saying that
if a chord diagram $D$ 
has an isolated chord  of framing 0 (i.e.,
a chord having two ends
adjacent along the circle of $D$) then $D=0$. 
\end{definition}

In the definition of the $1T$-relation even in the framed case
the chord should have framing zero.

\begin{definition}
By a {\em weight system} (resp.,
{\em framed weight system}) we mean
a linear function on chord diagrams
(resp., framed chord diagrams)
satisfying the $4T$-relations
(resp., {\em framed $4T$-relations})\footnote{In different sources, 
people may or may not require
the $1T$-relation to hold; later
on we shall mention the $1T$-relation
explicitly.}.
\end{definition}

For arc diagrams (framed or not
framed) there is a natural way
to multiply $A, B\mapsto A\cdot B$
by attaching the head of the
line for $A$ to
the tail of the line for $B$.

Certainly, this definition can
be extended when we take the
quotient by any relations
(1T or 4T), so we get
${\cal A}, {\cal A}^{f},
{\cal A}^{f*}$, and
a natural isomorphism
${\cal A}^{f}$ between ${\cal A}^{f*}$

For chord diagrams (on the circle)
one can define the multiplication
by cutting two circles at some
points (distinct from the end of
any chord) and attaching to one
another with respect to the orientation.
The cutting operation is {\em not} well defined,
but if we take unframed chord diagrams
modulo the 4T-relation, it becomes well defined (see, e.g.,
\cite{Bar-Natan}). Once cutting
becomes well defined, we get
a well defined multiplication.

In particular this means
that {\em chord diagrams modulo the $4T$-relation form an algebra}, moreover, this algebra is {\em commutative}, since
the product $AB$ and the product $BA$
can be chosen identical.
This algebra is naturally graded by
the number of chords\footnote{There is
also a coalgebra structure on
these spaces which is well defined
everywhere}.

For framed chord diagram the 
cutting operation is not well-defined,
however, the question about commutativity
of framed arc diagrams is still open,
\cite{IM}.

\section{Vassiliev's invariants}

Given an oriented manifold $M^{3}$,
we
consider knots in $M^{3}$ and standardly define singular knots, derivatives,
and invariants of order $n$ for knots
in $M^{3}$ in a standard way.

For simplicity we shall deal
with knots (and singular knots) which are $\Z_{2}$-homologically
trivial in $M$.

\begin{definition}
Let $n$ be a non-negative integer.
By a {\em singular knot} of order
$n$ in a $3$-manifold $M^{3}$ we mean a map $k: S^{1}\to M^{3}$ which is an embedding
everywhere except $n$ pairs of ``glued'' points
$a_{i},b_{i}\in S^{1}$, such that
$k(a_{i})=k(b_{i}), i=1,\cdots, n$, no other points are glued, and all 
intersections at $k(a_{i})$
are simple and transverse.
\end{definition}

Singular knots are considered up
to a natural isotopy (keeping all glued
points glued).
Each singular knot $K$ of order $n$
at each crossing $k_{i}=k(a_{i})=k(b_{i})$
admits two natural resolutions, see Fig. \ref{herefig},
which lead to two singular knots of 
order $n-1$.

\begin{figure}
\centering\includegraphics[width=200pt]{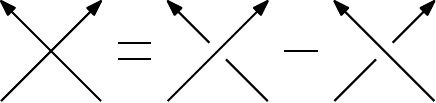}
\caption{The Vassiliev relation}
\label{herefig}
\end{figure}

One can smooth any number $j=1,\cdots, n$
of crossings in $2^{j}$ ways
which leads to a singular knot of order
$n-j$, in particular, for $j=n$
we get $2^{n}$ ways of getting 
non-singular knots.

With a singular knot of order $n$
we associate its {\em chord diagram}
(with $n$ chords)
by connecting by chords those points
on $S^{1}$ having the same image,
see Fig. \ref{trefle}.

\begin{figure}
\centering\includegraphics[width=200pt]{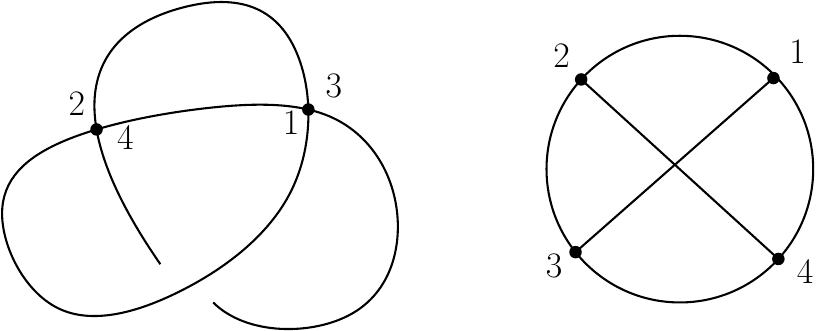}
\caption{The chord diagram
corresponding to a singular knot}
\label{trefle}
\end{figure}

\begin{definition}
Let $f$ be an invariant of knots in
$M$. We define the $n$-th derivative
of $f$ (denoted by $f^{(n)}$) as 
the function on singular knots with
$n$ crossings as an alternating sum defined as follows.
Given a singular knot ${\cal K}$ with 
crossings $c_{1},\cdots, c_{n}$;
we take all $2^{n}$ resolutions
$K_{s}=K_{s_1,\cdots,s_{n}}, s_{j}\in \{0,1\}$
and set

$$f^{(n)}(K)= \sum_{s} (-1)^{\sum s_{i}}
f(K_{s}) .$$

In particular, for a singular knot $K$ with one singular crossing the first derivative
$f'(K)$ is defined as $f(K_{0})-f(K_{1})$
where $K_{0}$ and $K_{1}$ are the two
resolutions of $K$ as in Fig. \ref{herefig}.
 
\end{definition}

For knots in $\R^{3}$ it is known that the symbol is a function
on the corresponding (unframed) chord diagram
(it does not change if we pass from
one singular knot to another leaving
singular crossings fixed).

The crucial observation 
(Vassiliev \cite{Vas})
is that such a function has to satisfy
the $1T$-and the $4T$-relations.
Interestingly the opposite is true
as well for the case of $\R^{3}$ (Kontsevich, \cite{Kontsevich}, see also Bar-Natan \cite{Bar-Natan}),
namely, each linear function on
unframed chord diagrams satisfying the
$1T$- and the $4T$-relations
is a symbol of some Vassiliev invariant; we shall discuss Vassiliev invariants and
Kontsevich integral in Section 4.

Framed knots in $M^{3}$ possess
much more topological information
which allows one to enhance chords
of the chord diagram with some
{\em framing}. 

Assume that $H_{1}(M^{3},\Z_{2})\neq 0$ and fix a cohomology class 
$\alpha \in H^{1}(M^{3},\Z_{2})$ 
which can be non-trivially evaluated
on $H_{1}(M^{3},\Z_{2})$. 

With each closed loop $l$ in $M^{3}$
we associate the element from $\alpha(l)\in\Z_{2}$.
Assume $\alpha[K]=0\in\Z_{2}$.

This allows one to define the
{\em framing} (parity) for singular crossings and chords of the chord
diagram for
any singular knot obtained from $K$: if a
chord $c$ of the chord diagram $K$
corresponds to two glued points
$c_{1}=c_{2}$ then it gives rise
to two ``halves'', $K'$ and
$K''$, and we evaluate the
class $\alpha$ on either of
them to define the parity
$p(c)=\alpha([K'])=\alpha([K''])$.
It is well defined since $\alpha([K])=0$.

\begin{definition}
We say that a framed chord diagram is
{\em realisable} in the manifold
$M^{3}$ if there is a singular
knot having this chord diagram.
\end{definition}

It is a well-known fact that 
{\em any symbol of an order $n$
Vassiliev invariant} of knots in
$\R^{3}$ is a weight function
satisfying the $1$-term relation,
see the section below about the
Kontsevich integral.

Before generalising this theory
to the case of {\em framed}
chord diagrams, it is important
to note that 
{\em not all framed diagrams are realisable} for $M^{3}$. Indeed, for example
if $H_{1}(M^{3},\Z_{2})=0$ (as
in the case of $\R^{3}$) then
all chords of all chord diagrams corresponding
to singular knots have framing $0$.

\begin{definition}
We say that an invariant $f$ of knots
in $M$ {\em has order $\le n$} if
its $(n+1)$-st derivative is $f^{(n+1)}=0$.
We say that $f$ is {\em a Vassiliev invariant of order $n$} if it has
order $\le n$ but not $\le (n-1)$.
\end{definition}

\begin{definition}
Let $f$ be a Vassiliev invariant of a knot in
$M^{3}$ of order $n$. The {\em symbol} of a Vassiliev invariant is its $n$-th derivative $f^{(n)}$.

\end{definition}

Having a singular knot $K$ of order $n$ and a Vassiliev invariant $f$ of order $n$,
the value of $f$ does not change if we perform a crossing switch (pass through a singular knot of order $n+1$).

So, the symbol of the Vassiliev invariant for knots in $M^{3}$ is well defined on
isotopy classes of singular knots modulo
the crossing switch relations (homotopy
of a singular knot).

It is now important to note that 
{\em the symbol of a Vassiliev invariant
for {\bf classical} knots (in $\R^{3}$)
is a well-defined function on chord diagrams since the above isotopy classes
are defined by chord diagrams}.

For arbitrary $M^{3}$ this may not be true.
For example, if we have a chord diagram with one solitary chord then the symbol of a Vassiliev invariant of order $1$ should be equal to $0$ on such a chord diagram
if one of the halves of the singular knot
is homotopically trivial, but it need not be so in general case.

Certainly, just one $\Z_{2}$-framing of a chord diagram coming from just one cohomology class $\alpha$ is not sufficient to restore 
the homotopy class of a singular knot\footnote{This can be done if
we impose some more accurate (homotopy) framing; we shall do it elsewhere.}.



\begin{definition} {\em A Vassiliev invariant $f$ of order $n$} for knots in $M^{3}$ is 
{\em good} with respect to the cohomology
class $\alpha\in H^{1}(M^{3},\Z_{2})$ if 
its value on singular knots of order $n$ is well defined on framed chord diagrams of order $n$ (does not depend on singular knot
itself but rather on the chord diagram).
\end{definition}

\begin{lemma}
Let $C_{1},C_{2},C_{3},C_{4}$
be four framed chord diagrams
taking part in the four-term relation.

Then if one of $C_{i}$ is realisable
then all four diagrams are realisable.
\end{lemma}
This lemma is easy: starting from
realisation of one $C_{i}$,
we can transform the singular knot {\em locally}
to get all other $C_{j}$.

Let us fix a cohomology class
$\alpha\in H^{1}(M^{3},\Z_{2})$ and consider
knots in $M^{3}$ such that $\alpha([K])=0$.

The main result of the present section is the following
\begin{theorem}
Let $v^{(n)}$ be the symbol of a good 
Vassiliev invariant $v$ of order
$n$ defined as a function on realisable framed
chord diagrams. 

Then it satisfies the framed $4T$-relation 
for all quadruples of {\em realisable}
chord diagrams.
Also, it satisfies the $1T$-relation
for realisable chord diagrams.
\label{mainthm}

Consequently,
each symbol of a good Vassiliev invariant of order $n$ is mapped to a function on the quotient space ${\cal C}^{f}_{n}$.
\end{theorem}

\begin{proof}

The proof is rather standard.
Indeed, for the $1T$-relation one
can construct two diagrams of a
singular knot $K$.
For the $4T$-relation we just
repeat Fig.14.1 from 
\cite{KnotTheory}.

Consider an invariant $v$ of order  $n$ and the
values of its symbol on these four knots. Vassiliev's
relation implies the relations shown in Fig.~\ref{F6}.

\begin{figure}[t]
\begin{center}

\hspace{-1cm}
\unitlength 0.9mm
\linethickness{0.8pt}
\begin{picture}(125,101.33)
\put(35,90){\vector(1,0){0.2}} \put(15,90){\line(1,0){20}}
\multiput(20,80)(0.12,0.21){67}{\line(0,1){0.21}}
\put(31.67,101.33){\vector(1,2){0.2}}
\multiput(29.33,97)(0.12,0.22){20}{\line(0,1){0.22}}
\put(25,100.67){\vector(-1,1){0.2}}
\bezier{116}(25,80)(35,90.33)(25,100.67)
\put(30,90){\circle*{1.33}} \put(25,90){\circle*{1.33}}
\put(35,65){\vector(1,0){0.2}} \put(15,65){\line(1,0){20}}
\put(35,40){\vector(1,0){0.2}} \put(15,40){\line(1,0){20}}
\put(25,65){\circle*{1.33}} \put(25,40){\circle*{1.33}}
\put(25,15){\circle*{1.33}} \put(20,55){\vector(1,2){10}}
\bezier{20}(25,55.33)(22.67,56.67)(23,59.33)
\put(25,75){\vector(1,1){0.2}} \bezier{76}(21,61)(17,68)(25,75)
\put(20,65){\circle*{1.33}} \put(20,30){\vector(1,2){10}}
\bezier{44}(25.33,30)(30,36)(30.67,39)
\put(25,50){\vector(-1,1){0.2}}
\bezier{44}(30.67,41)(29,47)(25,50) \put(28,46.67){\circle*{1.33}}
\put(15,15){\line(1,0){3.33}} \put(21.33,15){\vector(1,0){13.67}}
\put(20,5){\vector(1,2){10}} \put(25,25){\vector(1,1){0.2}}
\bezier{112}(25,5)(15,15)(25,25) \put(65,90){\circle*{1.33}}
\put(105,90){\circle*{1.33}} \put(65,40){\circle*{1.33}}
\put(105,40){\circle*{1.33}} \put(65,15){\circle*{1.33}}
\put(105,15){\circle*{1.33}}
\multiput(60,80)(0.12,0.21){67}{\line(0,1){0.21}}
\multiput(100,80)(0.12,0.21){67}{\line(0,1){0.21}}
\put(71.67,101.33){\vector(1,2){0.2}}
\multiput(69.33,97)(0.12,0.22){20}{\line(0,1){0.22}}
\put(111.67,101.33){\vector(1,2){0.2}}
\multiput(109.33,97)(0.12,0.22){20}{\line(0,1){0.22}}
\put(75,40){\vector(1,0){0.2}} \put(55,40){\line(1,0){20}}
\put(115,40){\vector(1,0){0.2}} \put(95,40){\line(1,0){20}}
\put(55,15){\vector(1,0){20}} \put(95,15){\line(1,0){3.33}}
\put(101.33,15){\vector(1,0){13.67}} \put(22,8.67){\circle*{1.33}}
\put(75,90){\vector(1,0){0.2}} \put(55,90){\line(1,0){20}}
\bezier{44}(65,80)(69.33,84.67)(70,89.33)
\put(65,100){\vector(-3,4){0.2}}
\bezier{40}(70.33,91.33)(68.33,95.67)(65,100)
\put(105,100){\vector(-1,1){0.2}}
\bezier{116}(105,80)(115.67,90)(105,100)
\put(95,90){\line(1,0){14}} \put(112,90){\vector(1,0){3}}
\put(60,55){\vector(1,2){10}} \put(100,55){\vector(1,2){10}}
\put(65,65){\circle*{1.33}} \put(105,65){\circle*{1.33}}
\put(55,65){\vector(1,0){20}}
\bezier{20}(65,55)(62.67,57.33)(63,59.33)
\bezier{20}(61.33,60.67)(59.33,62.67)(60,64.33)
\put(65,75){\vector(1,1){0.2}}
\bezier{44}(60,65.67)(60.33,70.33)(65,75)
\bezier{20}(105,55.33)(103.67,56.67)(103,59.33)
\put(105,75){\vector(1,1){0.2}} \bezier{76}(101,61)(95,70)(105,75)

\put(95,65){\line(1,0){3}} \put(100,65){\vector(1,0){15}}

\put(60,30){\vector(1,2){10}}
\bezier{44}(65,30)(70,34.67)(70,39)
\bezier{24}(70,41)(70.67,44.33)(68.67,45)
\put(65,50){\vector(-3,4){0.2}}
\bezier{16}(67,47)(66.67,48.67)(65,50)
\bezier{44}(105,30)(109.67,36)(110.33,39)
\put(104.67,50){\vector(-1,1){0.2}}
\bezier{44}(110.33,41)(106,47)(104.67,50)
\put(100,30){\line(1,2){7}} \put(109,46.33){\vector(1,2){3}}
\put(60,5){\vector(1,2){10}}
\bezier{16}(65,5.33)(63,6.33)(62.67,8.33)

\put(65,25){\vector(3,4){0.2}}
\bezier{76}(60.67,10)(59,12)(59,14)
\bezier{76}(59,16)(63,20)(65,25)

\put(105,25){\vector(1,1){0.2}}
\bezier{112}(105,5)(95,15)(105,25)
\put(99.33,3.67){\line(1,2){2}}
\put(102.67,10.33){\vector(1,2){7.33}}
\put(10,90){\makebox(0,0)[cc]{$v^{(n)}($}}
\put(38.33,90){\makebox(0,0)[cc]{$)$}}
\put(50,90){\makebox(0,0)[cc]{$v^{(n-1)}($}}
\put(90,90){\makebox(0,0)[cc]{$v^{(n-1)}($}}
\put(78.33,90){\makebox(0,0)[cc]{$)$}}
\put(118.33,90){\makebox(0,0)[cc]{$)$}}
\put(42,90){\makebox(0,0)[cc]{$=$}}
\put(82,90){\makebox(0,0)[cc]{$-$}}
\put(125,90){\makebox(0,0)[cc]{$=a-b$}}
\put(10,65){\makebox(0,0)[cc]{$v^{(n)}($}}
\put(10,40){\makebox(0,0)[cc]{$v^{(n)}($}}
\put(10,15){\makebox(0,0)[cc]{$v^{(n)}($}}
\put(38.33,65){\makebox(0,0)[cc]{$)$}}
\put(38.33,40){\makebox(0,0)[cc]{$)$}}
\put(38.33,15){\makebox(0,0)[cc]{$)$}}
\put(50,65){\makebox(0,0)[cc]{$v^{(n-1)}($}}
\put(50,40){\makebox(0,0)[cc]{$v^{(n-1)}($}}
\put(50,15){\makebox(0,0)[cc]{$v^{(n-1)}($}}
\put(90,65){\makebox(0,0)[cc]{$v^{(n-1)}($}}
\put(90,40){\makebox(0,0)[cc]{$v^{(n-1)}($}}
\put(90,15){\makebox(0,0)[cc]{$v^{(n-1)}($}}
\put(78.33,65){\makebox(0,0)[cc]{$)$}}
\put(78.33,40){\makebox(0,0)[cc]{$)$}}
\put(78.33,15){\makebox(0,0)[cc]{$)$}}
\put(118.33,65){\makebox(0,0)[cc]{$)$}}
\put(118.33,40){\makebox(0,0)[cc]{$)$}}
\put(118.33,15){\makebox(0,0)[cc]{$)$}}
\put(42,65){\makebox(0,0)[cc]{$=$}}
\put(42,40){\makebox(0,0)[cc]{$=$}}
\put(42,15){\makebox(0,0)[cc]{$=$}}
\put(82,65){\makebox(0,0)[cc]{$-$}}
\put(82,40){\makebox(0,0)[cc]{$-$}}
\put(82,15){\makebox(0,0)[cc]{$-$}}
\put(125,65){\makebox(0,0)[cc]{$=c-d$}}
\put(125,40){\makebox(0,0)[cc]{$=c-a$}}
\put(125,15){\makebox(0,0)[cc]{$=d-b$}}
\end{picture}
\end{center}
\vspace{-0.6cm} \caption{The same letters  express $v^{(n-1)}$ for
isotopic long knots} \label{F6}
\end{figure}

Obviously, $$(a-b)-(c-d)+(c-a)-(d-b)=0.$$ In order to get singular
knots, one should close the fragments drawn in Fig.\ref{F6}.
In Fig. \ref{F6} we have
four singular knots of order
$n$, where two singular crossings
are present and the remaining $n-2$
ones are outside the picture
(they are the same for all terms).

We have to close these up to get
four chord diagrams. 

They will differ exactly in the neighbourhood of the two chords.
A careful look at the framings 
gives us exactly the framed $4T$-relation.

The proof of the $1T$-relation is left to the
reader.

\end{proof}

There are a lot of nice weight systems coming from parity; even in degree $1$ we can count odd chords; one may expect to have nice and easy combinatorial formulas for Vassiliev invariants of virtual knots and for knots in complements to other knots.

\section{The Kontsevich integral}

For knots in $\R^{3}$ the construction
of the Vassiliev invariant from its symbol is realised
as follows.

Let $Z$ be the Kontsevich integral 
\cite{Kontsevich} defined as a series
of linear functions on framed chord diagrams.

The {\em preliminary Kontsevich integral}
$Z(\cdot)$ is defined by the following
equation.

\begin{equation}\label{preKontsevich}
Z(K)=\sum_{m=0}^{\infty} { {1\over {(2\pi i)^m}}
\int_{c_{min}<t_{1}<\dots<t_{m}<c_{max} \atop
t_{j}\mbox{\small{non-critical}}}{
\sum_{P=\{(z_{j},z_{j}^{'})\}}{(-1)^{\downarrow}D_{P}
\bigwedge_{j=1}^{m} {{dz_{j}-dz_{j}^{'}}\over {z_{j}-z_{j}^{'}}} }
}}.
\end{equation}
We decree the coefficient of the ``empty'' chord
diagram to be equal to one.

In (\ref{preKontsevich}) $K$ is a knot
in $\R^{3}=\cC_{z}\times \R^{1}_{t}$ which
is assumed to be a Morse knot with
respect to the height function $t$;
 $c_{min}$
and $c_{max}$ are minimal and maximal
values of $t$
on the knot $K$, and $t_{k},k=1,\cdots, m$ are other critical values; for each
non-critical value $t_{j}$ we choose a pair of points complex points $\{z_{j},z'_{j}\}\times t_{j}$ on the knot; having
choses such points and pairing them 
by chords,
we get the chord diagram $D_{P}$ corresponding to the configuration
$P$. As all $t_{k}$ move,
pairs $z_{k},z'_{k}$ move accordingly,
and we can integrate the form
$\bigwedge_{j=1}^{m} {{dz_{j}-dz_{j}^{'}}\over {z_{j}-z_{j}^{'}}} .$ 

The normalised version of the Kontsevich integral
$I(\cdot)$ is defined by norming 
$I(K)=Z(K)/Z(\infty)^{c/2-1}$,
where $\infty$ is the unknot in $\R^{3}$
having two maxima and two minima with
respect to $z$-coordinate, and
$c$ is the number of extrema of the knot $K$.

Both $Z$ and $I$ are valued in 
linear combinations of chord diagrams.

Denote by $I(K)_{n}$ the $n$-th
graded part of $I(K)$, i.e.,
the linear combination of chord diagrams
of chord diagrams from
$I(K)$ having $n$ chords.

The Kontsevich theorem (for knots
in $\R^{3}$) states that for a weight system $w$ of order $n$,
the invariant
$$w(\{I(K\}_{n})$$ 
is a Vassiliev invariant with symbol $w$.

Hence, we have a very nice formula for
{\em all Vassiliev invariants of classical knots}.

The universal invariant for the
Vassiliev invariant in thickened
surfaces does exist, 
\cite{AMR,HM} but not quite
in the way given above.

Can we integrate any weight system 
``by hand''? This problem has been
studied by many authors.

A fundamental theorem due to M.N.Goussarov, (see \cite{GPV}) says that {\em any Vassiliev invariant admits
a combinatorial formula in 
terms of chord diagrams}. We are not going
to give an exact definition of a combinatorial
formula and advise the reader to enjoy
Goussarov's orginal paper!

\section{Further directions}

How can we upgrade such formulas 
by using framed chord diagrams?
Where to find such weight systems?
How to integrate them? For the case of
$\Z_{2}$ this is done in \cite{CDL}.

For classical knots, all chord diagrams
have all chords with framing $0$.

However, it is well known that
the Kontsevich integral can be generalised
for various manifolds of $S_{g}\times I$,
where $S_{g}$ is a sphere with $g$ handles.

Certainly it works for
$\R^{2}\backslash \{*,\cdots, *\}\times I$ as well.

\subsection{What to do for knots in complements to knots?}

The idea is very simple: we
just double the knot $K$ to get a $2$-component
link $L_{1}\sqcup L_{2}$, and then consider
$L_{1}$ as a knot in the complement
$\R^{3}\backslash L_{2}$.

Our hope is to enhance the notion
of the Vassiliev knot invariants
(for 1-component knots by using
this construction, parity, framed chord diagrams, and the corresponding version of the Kontsevich integral).

In particular, we can do the Kontsevich
integral for the thickening of the 
plane with $2n$ punctures
$$P_{n}=\R^{2}\backslash \{*,\cdots, *
\} \times \R^{1}.$$

We have
$H_{1}(P_{n})=\Z^{n}$ generated by homology classes $a_{j},j=1,\cdots, n$ 
corresponding to punctures. 
Let $K$ be a knot in $P_{n}$.
We have $[K]=\sum_{i=1}^{n} \alpha_{i}a_{i}$. We say that $K\subset P_{n}$ is {\em even}
if $\sum_{i}\alpha_{i}\equiv 0\; \mbox{(mod 2)}$.

Consider the Kontsevich integral
for $K\subset P_{n}\subset \R^{3}$. Each chord diagram $D_{P}$ in (\ref{preKontsevich})
naturally acquires the parity,
hence we can upgrade
$Z(K)$ to its framed version 
$Z^{f}(K)$ and $I(K)$ to $I^{f}(K)$ by replacing all terms in $Z(K)$ and
$Z(\infty)$ by their framed version.

After that, the program is as follows:
for a knot $K$ in the complement 
$\R^{3}\backslash K'$
to the knot $K'\subset \R^{3}$
such that the linking number
$lk(K,K')$ is even, 
we can assume that $K'$ is in
general position with respect to
some horizontal plane $P$.

Certainly, $K\cap P$ consists
of some odd number of points (say, $2n$) and
this number may vary during the isotopy of $K$.
Let us consider
$K$ in $P\times{[-\varepsilon,+\varepsilon]}\backslash K'$,
where ${\varepsilon}$ is a small number
such that the neighbourhood of $P$
intersects $K'$ at $2n$ segments.

Hence we can take $I^{f}(K,K')=I^{f}(K)$
with respect to this $P_{n}$.

We conjecture that $I^{f}(K,K')$
will be an invariant of $K'$.

Certainly, $I'(K,K')$ allows us to
recover $I(K)$ by forgetting
the framing of all chords in chord
diagrams.

To get {\em the other knot $K'$}
we can just take the {\em doubling
of $K$}.

\subsection{Doubling the knot}

Let ${\cal K}$ be the set of knots
in $\R^{3}$.
Consider the sets ${\cal L'},{\cal L}$ of ordered two-component links $L_{1}\sqcup L_{2}$
where:

\begin{enumerate}

\item
In ${\cal L'}$ links are considered up to isotopy,
\item 
In ${\cal L}$ the first component $L_{1}$ is considered up to homotopy in the complement $\R^{3}\backslash L_{2}$
whence $L_{2}$ is considered up to isotopy.
\end{enumerate}

We denote this composite  map
$f:{\cal K}\to {\cal L'}\to {\cal L}$ by
$f: K\mapsto (L_{1},L_{2})\in {\cal L}$.

\begin{theorem}
The map $f$ is well defined.
\end{theorem}

Indeed, $f$ is a composite of 
two well defined maps: first
we take the satellite operation and
then we weaken the equivalence relation.
Hence, any topological invariants of
the double of $K$  yield invariants of $K$.

We can consider oriented or non-oriented
links. In the sequel, all knots and 
links are thought to be oriented
unless specified otherwise.

Moreover, the component $L_{2}$ being
a knot lives not just in $\R^{3}$ though
in the complement to $L_{1}$ hence,
for $L_{2}$ we have more tools for constructing invariants. In particular, one can
use the parity coming from the 
$\Z_{2}$-homology group $H_{1}(\R^{3}\backslash
L_{1},\Z_{2})$.
The general concept of parity was established in
\cite{Parity2010}.

In \cite{Parity2010} the author
introduced the notion of {\em parity}
into knot theory and low-dimensional
topology. Roughly speaking, when
one can distinguish between 
{\em odd} and {\em even} types
of crossings which satisfy some
natural axioms when moves are applied
to knot diagrams, then one can
construct very powerful invariants
and enhance some known ones.
In many cases parity originates
from $\Z_{2}$-homology \cite{IMN}.

Then $L_{1}$ (more precisely,
$H_{1}(L_{1},\Z_{2})$) will be a
source of parity for 
the knot $L_{2}$ living in the complement
$\R^{3} \backslash L_{1}.$

\section{Epilogue. More framings and more problems. Conjectures}

First, we conjecture that the multiplication in $\Z_{2}$-framed
chord diagram algebra is {\bf not}
well defined and that 
framed chord diagrams 
are {\bf not} invertible modulo
$1T$- and framed $4T$-relations.

Besides, we conjecture that
the above composite map will lead to {\em generalisations}
of Vassiliev invariants
by using framed diagrams
and the structure on them
which may lead to solution
of some nice problems for classical knots
(knot invertibility etc). 

Certainly, we simplified the situation
drastically when we took only $\Z_{2}$-framing. Indeed, if our knot $K$
lives in a $3$-manifold $M^{3}$
with a non-trivial fundamental group
then it would be more reasonable to 
associate some integer homology or
first homotopy to chords of the corresponding chord diagram.

Certainly, this could lead to a richer linear
space of chord diagrams (coalgebra), but the one with
$\Z_{2}$-framing is already interesting and not
quite well studied (say, is the product
well defined).

Do we have an enhancements of a knot $K_{2}$ in a complement to another knot $\R^{3}\backslash K_{1}$? To this end, one should
first investigate the set of {\em moves}
for representing isotopy classes of pairs of knots $(K_{1},K_{2})$ such that $K_{1}$ lives in a plane and $K_{2}$ goes through
this plane transversally in some $2n$ points.

Some steps towards it were investigated in
\cite{GM}.

However, we do not have a set of such moves, but if we take just the above parity which counts how many times some part of $K_{2}$ winds around $K_{1}$, then this parity is visible from the plane $P_{n}$ as described above. 

Certainly, this parity does not change
if $K_{1}$ undergoes {\em homotopy} rather than isotopy, which motivated us to construct the subsection about doubling.

One possible outcome can be formulated as follows.

We say that the link $L_{1}\sqcup L_{2}$ is {\em standard}
if there is a plane $\R^{2}$ such that
for some small $\varepsilon$ the intersection
$L_{1}\cap (\R^{2}\times [-\varepsilon,+\varepsilon])$ 
consists of vertical segments
$\{*\}\times [-\varepsilon,+\varepsilon]$
(say, $2n$ ones) and $L_{2}\subset
\R^{2}\times [-\varepsilon,+\varepsilon]$.

One possible program would be to look for a 
``universal $\Z_{2}$-framed classical finite-type invariant'' $f$ taking  knots in
$(\R^{2}\backslash \{\mbox{evenly many points}\})\times I$
to $\Q$
(defined once for any even number of 
points in a consistent way)
such that the composition of the maps below
gives an invariant of the knot $K$:

First, take
$K\to L(K) = L_{1}\sqcup L_{2} \to L_{2} \subset \R^{3}\backslash L_{1}$.

Now put $L_{1}\sqcup L_{2}$ in the standard form.
Then $L_{2}\subset \R^{3}\backslash L_{1}$ gives
rise to $L_{2} \subset (\R^{2}\backslash \{\mbox{evenly many points}\})\times I$.

After that we take an invariant $f$
taking 
$L_{2} \subset (\R^{2}\backslash \{\mbox{evenly many points}\})\times I$ to $\Q$. Can
we make the whole composition to be an invariant
of $K$?

An important problem here is how to understand
the structure of ${\cal L}$.

We conjecture two standard links $(L_{1}\sqcup
L_{2})$ and $(L'_{1}\sqcup L'_{2})$  are equivalent
as elements of ${\cal L}$ (maybe, for different
$n$)
if and only if they
can be transformed to each other by:

\begin{enumerate}

\item Isotopy of the first component
$L_{1}$ in the thickened punctured plane 
$P\times [-\varepsilon,+\varepsilon]\backslash L_{2}$;

\item Homotopies of $L_{2}$ in the complement
to $P\cup L_{1}$;

\item Addition/removal of pairs of points
in $P$ as $L_{2}$ goes through it.

\end{enumerate}

\end{document}